\documentclass[11pt]{amsart}
\usepackage{mathptmx,microtype,libertine}
\usepackage[T1]{fontenc}
\usepackage{amsmath, amscd, amsthm} 
\usepackage{amssymb, amsfonts} 
\usepackage{stmaryrd}
\usepackage{enumerate}
\usepackage[svgnames]{xcolor} 
\usepackage{color}
\usepackage[margin=1.5in]{geometry}
\usepackage{mathrsfs}
\usepackage{booktabs}
\usepackage{aliascnt}
\usepackage{mathrsfs}
\usepackage{hyperref}
 \definecolor{dark-red}{rgb}{0.4,0.15,0.15}
\hypersetup{
    colorlinks, linkcolor=dark-red,
    citecolor=DarkBlue, urlcolor=MediumBlue
}
\usepackage{mathrsfs}
\usepackage[
textwidth=3cm,
textsize=small,
colorinlistoftodos]
{todonotes}

\usepackage{tikz}
\usepackage{tikz-cd}

\DeclareMathOperator{\Tr}{Tr}

\DeclareMathOperator{\cosat}{cosat}
\DeclareMathOperator{\sat}{sat}

\DeclareMathOperator{\reft}{reflect}
\DeclareMathOperator{\coreft}{coreflect}

\newcommand{\op}{{\mathrm{op}}}
\newcommand{\refl}{{\mathrm{ref}}}
\newcommand{\Sub}{\operatorname{Sub}}
\newcommand{\SubMon}{\operatorname{SubMon}}

\newcommand{\id}{\operatorname{id}}

\newcommand{\Fac}{\operatorname{Fac}}
\newcommand{\End}{\operatorname{End}}

\newcommand{\clop}{\operatorname{\overline{\End}}}

\usepackage[T1]{fontenc}

\numberwithin{equation}{section} 

\theoremstyle{plain}

\newaliascnt{theorem}{equation}  
\newtheorem{theorem}[theorem]{Theorem}  
\aliascntresetthe{theorem}

\newaliascnt{dodeca}{equation}  
  
\aliascntresetthe{dodeca}

 \theoremstyle{definition}

\newaliascnt{prop}{equation}  
\newtheorem{prop}[prop]{Proposition}
\aliascntresetthe{prop}

\newaliascnt{lemma}{equation}  
\newtheorem{lemma}[lemma]{Lemma}
\aliascntresetthe{lemma}

\newaliascnt{corollary}{equation}  
\newtheorem{corollary}[corollary]{Corollary}
\aliascntresetthe{corollary}

\newaliascnt{claim}{equation}  

\aliascntresetthe{claim}

\newaliascnt{conjecture}{equation}  

\aliascntresetthe{conjecture}

\newaliascnt{question}{equation}  

\aliascntresetthe{question}

\newaliascnt{defn}{equation}  
\newtheorem{defn}[defn]{Definition}
\aliascntresetthe{defn}

\newaliascnt{example}{equation}  

\aliascntresetthe{example}

\theoremstyle{remark}

\newaliascnt{remark}{equation}  
\newtheorem{remark}[remark]{Remark}
\aliascntresetthe{remark}

\newaliascnt{convention}{equation}  

\aliascntresetthe{convention}

\theoremstyle{plain}

\begin{document}
\title{Combinatorics of factorization systems on lattices}

\author[Bose]{Jishnu Bose}
\address{University of Southern California}
\email{jishnubo@usc.edu}

\author[Chih]{Tien Chih}
\address{Oxford College of Emory University}
\email{tien.chih@emory.edu}

\author[Housden]{Hannah Housden}
\address{Vanderbilt University}
\email{hannah.housden@vanderbilt.edu}

\author[Jones]{Legrand Jones II}
\address{Indiana University Bloomington}
\email{legjones@iu.edu}

\author[Lewis]{Chloe Lewis}
\address{University of Wisconsin-Eau Claire}
\email{lewischl@uwec.edu}

\author[Ormsby]{Kyle Ormsby}
\address{Reed College}
\email{ormsbyk@reed.edu}

\author[Rose]{Millie Rose}
\address{University of Kentucky}
\email{math@milliero.se}

\begin{abstract}
We initiate the combinatorial study of factorization systems on finite lattices, paying special attention to the role that reflective and coreflective factorization systems play in partitioning the poset of factorization systems on a fixed lattice. We ultimately uncover an intricate web of relations with such diverse combinatorial structures as submonoids, monads, Moore systems, transfer systems (from stable equivariant homotopy theory), and poly-Bernoulli numbers.
\end{abstract}

\maketitle

\setcounter{tocdepth}{1}
\tableofcontents

\section{Introduction}\label{sec:intro}
Orthogonal factorization systems are foundational structures in category theory originally intorduced by Mac Lane \cite{MacLane}. They allow every morphism in a category to be (essentially) uniquely factored as a `left' morphism followed by a `right' morphism, with surjective-injective factorizations of functions between sets serving as a prototype. Weak factorization systems play a similar role in homotopy theory, but relax the uniqueness condition; they arise in the foundational work of Quillen \cite{quillen} on model structures.

Orthogonal and weak factorization systems are identical notions in a poset category, where there is at most one morphism between any pair of objects. The present work studies the combinatorics of factorization systems in this context. The authors' interest in such structures arose from the connection between factorization systems on subgroup lattices and transfer systems from stable equivariant homotopy theory exposed in \cite{fooqw}. While these applications remain top-of-mind, we have chosen to center factorization systems due to their historical precedence and their link to diverse topics of independent mathematical interest. Indeed, factorization systems on a lattice $P$ are closely related to (co)monads on $P$, closure and interior operators on $P$, and submonoids of $(P,\vee)$ and of $(P,\wedge)$ (where $\vee$ and $\wedge$ are the join and meet operators, respectively).

Our primary connection to these other topics is a well-developed theory of characteristic and cocharacteristic functions for factorization systems. This enhances and dualizes the theory of characteristic functions for transfer systems developed in \cite{REUChar}. The presentation in this context is also particularly natural: given a factorization system $(L,R)$ on a lattice $P$, we have $P$-endomorphisms $\chi^{(L,R)}$ and $\lambda^{(L,R)}$ which take $x\in P$ to the factoring objects of the universal relations $0\to x$ and $x\to 1$, respectively, resulting in the diagram
\[
\begin{tikzpicture}
    \node (A) at (-2,0){$0$};
    \node (B) at (0,0){$x$};
    \node (C) at (2,0){$1$.};
    
    \node (D) at (-1,-1){$\chi^{(L,R)}(x)$};
    \node (E) at (1, 1){$\lambda^{(L,R)}(x)$};

    \draw[->] (A) to (B);
    \draw[->] (B) to (C);
    \draw[->, blue] (A) --node[below left]{$L$} (D);
    \draw[->, DarkRed] (D) --node[below right]{$R$} (B);
    \draw[->, blue] (B) --node[above left]{$L$} (E);
    \draw[->, DarkRed] (E) --node[above right]{$R$} (C);
\end{tikzpicture}
\]
In \autoref{prop:lambdaanti}, we demonstrate that $\chi,\lambda\colon \Fac P\to \End P$ are antitone maps of posets, and in \autoref{thm:fibers} we provide a full analyses of the images and fibers of these assignments.

Returning to our original motivation for this work, we note that our combinatorial analysis of factorization systems via (co)characteristic functions has immediate consequences for structured $G$-equivariant ring spectra, where $G$ is a finite Abelian group. In this scenario, work of Blumberg--Hill \cite{BlumbergHill}, Rubin \cite{Rubin}, and Franchere--Ormsby--Osorno--Qin--Waugh \cite{fooqw} implies that the homotopy category of $G$-$N_\infty$ operads is equivalent to  the poset of factorization systems on the lattice of subgroups of $G$. To the equivariant-minded reader, we emphasize that our results provide structural and combinatorial control over the homotopy theory of $N_\infty$ operads.

\subsection*{Overview}
In \autoref{sec:fac}, we review definitions and basic properties of (orthogonal) factorization systems and recall their equivalence with transfer systems from \cite{fooqw}. \autoref{sec:cochar} contains the main contribution of this paper, defining both characteristic $\chi$ and cocharacteristic $\lambda$ functions for factorization systems. In \autoref{thm:fibers}, we prove that these surject onto closure and interior operators, respectively, providing a fundamental link to other flavors of combinatorics and order theory. We also show that the fibers of $\chi$ and $\lambda$ are bounded intervals whose maxima (respectively, minima) are coreflective (respectively, reflective) factorization systems.

In \autoref{sec:crypto}, we briefly exhibit how (co)reflective factorization systems can be recast as a number of other structures. None of the content in this section is particularly original, but the results are scattered or folklore and we feel it will benefit researchers to collate them in one place. We hope that our factorization and transfer system perspective will inspire further advances.

\subsection*{Acknowledgments}
Our work on this project was supported by the AMS Mathematics Research Communities Program Homotopical Combinatorics and NSF grant DMS--1916439. K.O.~was also supported in part by NSF grant DMS--2204365.

The authors thank Scott Balchin and Ethan MacBrough for sharing their unpublished work (joint with K.O.) with the rest of this group. A number of the observations from \autoref{sec:crypto} --- including the first-observed bijection between between disklike transfer systems and closure systems via a Galois connection --- are due to those authors.

\section{Factorization and transfer systems on lattices}\label{sec:fac}
We briefly review necessary definitions and results about factorization systems and transfer systems.

\subsection{Factorization systems}
Loosely speaking, a factorization system on a category $\mathbf{C}$ consists of a pair of classes of morphisms $(L,R)$ so that every morphism factors essentially uniquely as a morphism in $L$ followed by a morphism in $R$. This is codified by a lifting criterion as follows:

\begin{defn}\label{def:lift}
    Given a category $\mathbf C$ and morphisms $i\colon A\to B$ and $p\colon X\to Y$ in $\mathbf C$, consider diagrams in $\mathbf C$ of the form

    \[
    \begin{tikzpicture}
    \node (A) at (0,2){$A$};
    \node (B) at (2,2){$X$};
    \node (C) at (0,0){$B$};
    \node (D) at (2,0){$Y$.};

    \draw[->] (A) --node[above]{\small $ f$} (B);
    \draw[->] (C) --node[below]{\small $g$} (D);
    \draw[->, DarkRed] (B) --node[right]{\small$p$} (D);
    \draw[->, blue] (A) --node[left]{\small$i$} (C);
    \draw[->, dashed] (C)--node[above left]{\small $\lambda$} (B);
\end{tikzpicture}
    \]
    If for all $f\colon A\to X$ and $g\colon B\to Y$ making the outer square commute, there exists a \textbf{lift} $\lambda$ making the triangles commute, then $i$ is said to satisfy the \textbf{left lifting property} with respect to $p$, and $p$ satisfies the \textbf{right lifting property} with respect to $i$.
\end{defn}

Since we will specifically be working with poset categories and not general categories, we can say that for any given $A, B, X, Y\in P$, then there is at most one choice for morphisms $i,p,f,g, \lambda$.  Suppose then that $A\leq X, B\leq Y$, then the above square commutes, and $B\leq X$ if and only if the relations $A\leq B, X\leq Y$ satisfy the respective lifting properties with each other. 

\begin{defn}\label{def:liftclasses}
    Let $\mathbf{C}$ be a category and let $M$ denote a class of morphisms of $\mathbf C$.  Then we define
    \[^\bot M:=\{i:i \text{ satisfies the left lifting property for each }p\in M\}\] and 
    \[ M^\bot:=\{p:p  \text{ satisfies the right lifting property for each }i\in M\}.\]
    Note that for another class of morphisms $N$, we have $M\subseteq {}^\bot N$ if and only if $N\subseteq M^\bot$ and we write $N\perp M$ when this holds.
\end{defn}

\begin{defn}\label{def:WFS}
    Given a category $\mathbf C$, a \textbf{weak factorization system} on $\mathbf C$ is a pair of collections of morphisms $(L, R)$, such that:

    \begin{itemize}
        \item For each morphism $f$ of $C$, we have that $f=pi$ for some $i\in L, p\in R$.
        \item $L= {}^\bot R$ and $R=L^\bot$.
    \end{itemize}

    For a pair $(L, R)$ where the lift $\lambda$ from \autoref{def:lift} is always unique, we say that $(L, R)$ is an \textbf{orthogonal factorization system}.
\end{defn}

Note that in a poset category $P$, there is at most one morphism between any pair of objects, so weak factorization systems are automatically orthogonal. We write $\Fac(P)$ for the collection of all factorization systems on $P$ and give it a partial order with $(L,R)\le (L',R')$ if and only if $R\subseteq R'$ (equivalently, $L'\subseteq L$).

\subsection{Transfer systems}

Originally introduced by Rubin \cite{Rubin}, transfer systems are a combinatorial structure on a subgroup lattice $\Sub(G)$ encoding the homotopy category of $G$-$N_\infty$ operads from equivariant infinite loop space theory. In this work, we replace $\Sub(G)$ with an arbitrary finite lattice and ignore the conjugation action on $\Sub(G)$ needed in non-Abelian equivariant contexts. For a recent survey of both the combinatorics of transfer systems and their relevance to equivariant homotopy theory, the reader should consult \cite{notices_transfer}. 

\begin{defn}\label{def:transfersystem}
Given a finite lattice $P$, a \textbf{transfer system} $R$ is a partial order of $P$ that refines the underlying partial order of $P$ (\emph{i.e.}, $x~R~y\implies x\le y$) and whenever there are $x,y,z\in P$ such that $x\leq z$ and $y~R~z$, it follows that $x\wedge y~R~y$. 
\end{defn}

Considering $P$ as a category, this definition says that a transfer system is a wide subcategory of $P$ that is stable under pullbacks. Using this viewpoint, \cite{fooqw} showed that the transfer systems $R$ of a finite lattice $P$ are in fact the right morphisms of factorization systems of $P$ as a category.

We write $\Tr(P)$ for the collection of all transfer systems on $P$ and partially order it by $R\le R'$ if and only if $x~R~y\implies x~R'~y$. If we view $R$ and $R'$ as collections of morphisms, this is just the inclusion ordering.

\begin{theorem}[\cite{fooqw}]\label{thm:fooqw}
Let $P$ be a finite lattice.  Then there are inverse poset isomorphisms
\[
\begin{aligned}
  \Fac(P) &\longleftrightarrow \Tr(P)\\
  (L,R)&\longmapsto R\\
  ({}^\perp R,R)&\mathbin{\text{\reflectbox{$\longmapsto$}}} (L,R).
\end{aligned}
\]
\end{theorem}

By \cite{Rubin}, every transfer system induced by a linear isometries is operad is saturated, a notion we define presently.

\begin{defn}\label{defn:saturated}
A transfer system $R$ on a finite lattice $P$ is \textbf{saturated} when $x~R~y\le z$ and $x~R~z$ implies that $x~R~z$. This is equivalent to $R$ satisfying the 3-for-2 property.
\end{defn}

Prior work has been done to enumerate saturated transfer systems  on certain families of modular lattices.  We review these results here.

\begin{defn}
A \textbf{modular lattice} is a lattice $(M, \leq )$ where all $a, b, x\in M$ satisfy the \textbf{modular law}:
\[
a \leq b \implies a \vee (x \wedge b) = (a \vee x) \wedge b.
\]
\end{defn}

It is well-known that a lattice is modular if and only if it contains no pentagonal sublattices or, equivalently, satisfies the \emph{diamond isomorphism theorem}: $[x\wedge y,y]\cong [x,x\vee y]$ via $z\mapsto z\vee y$ and $z\wedge y\mapsfrom z$. From the perspective of equivariant homotopy theory, subgroup lattice of Abelian (or Dedekind) groups provide important examples of modular lattices.

\begin{defn}[\cite{REUChar}]\label{defn:satcov}
Let  $(M, \leq )$ be a modular lattice and let $Q\subseteq {\le}$ be a subset of covering relations\footnote{A relation $x\le y$ is called \emph{covering} when $x<y$ and $x\le z\le y$ implies $z=x$ or $y$; in other words, covering relations are minimal non-identity relations in $(M,\le)$.} for $M$. We call $Q$ a \textbf{saturated cover} when the following conditions hold:
\begin{enumerate}
    \item[(1)] For $x, y\in M$, if $x ~Q ~(x \vee y)$ then $(x \wedge y) ~Q~ y$.
    \item[(2)] Suppose that $x$ and $y$ are covered by $x\vee y$ and cover $x\wedge y$ (we call such a tetrad a \emph{covering diamond}); if three of the four covering relations between $x, y, x \wedge y$, and $x \vee y$ are in $Q$, then the fourth covering relation is in $Q$ as well.
\end{enumerate} 
\end{defn}

In \cite{REUChar}, a characterization of the saturated transfer systems were given in terms of saturated covers.

\begin{theorem}[\cite{REUChar}]\label{thm:matchstick}
Let $Q$ be a subset of covering relations on a finite modular lattice $(M, \leq )$. Then, $Q$ is the set of covering relations within a saturated transfer system if and only if $Q$ is a saturated cover, and this correspondence provides a bijection between saturated covers and saturated transfer systems on $M$.
\end{theorem}

Matchstick games were used by Hafeez--Marcus--Ormsby--Osorno \cite{hmoo} to count saturated transfer systems on rectangular lattices; see  \autoref{rmk:pB}. In \cite{REUChar}, the authors complete a similar enumeration on lattices of the form $[2]^{*n}$ --- that is, the $n$-fold parallel composition of the chain $[2]=\{0<1<2\}$ with itself; this lattice has a minimal element, maximal element, and $n$ intermediate incomparable elements.

To conclude this section, we isolate one additional class of transfer systems that will play an important role in our work.

\begin{defn}\label{defn:disklike}
We say that a transfer system $R$ on a finite lattice $P$ is \textbf{generated} by a subset $S$ of $\le$ when $R$ is the smallest transfer system containing the relations $S$. In this case, we write $R = \langle S\rangle$.

A transfer system $R$ on $P$ is \textbf{disklike} (or \textbf{cosaturated}) when it is generated by relations of the form $x~R~1$, \emph{i.e.},
\[
  R = \langle R/1\rangle.
\]
\end{defn}

In \cite{Rubin}, Rubin shows that every transfer system induced by a Steiner (or little disks) operad is disklike, explaining the terminology. While it is not obvious, saturated and disklike transfer systems participate in a formal duality which we make explicit in \autoref{cor:satdisk}.

\section{Characteristic and cocharacteristic functions}\label{sec:cochar}

We now come to the main construction of this paper, that of the characteristic function $\chi^{(L,R)}$ and cocharacteristic function $\lambda^{(L,R)}$ of a factorization system $(L,R)$. When transported to transfer systems via \autoref{thm:fooqw}, we will see that $\chi^{(L,R)}$ recovers the characteristic function $\chi^R$ of \cite{REUChar}.

\begin{remark}
We note that many of the below constructions can be made in the context of a general bicomplete category $\mathbf C$, but for ease of exposition we specialize to (the category induced by) a finite lattice $P$ throughout.
\end{remark}

\begin{defn}
Suppose $P$ is a finite lattice with minimal element $0$ and maximal element $1$. Let $(L,R)$ be a factorization system on $P$. The \textbf{characteristic function} $\chi^{(L,R)}\colon P\to P$ and \textbf{cocharacteristic function} $\lambda^{(L,R)}\colon P\to P$ are the maps taking an object $x$ of $P$ to the factoring objects $\chi^{(L,R)}(x)$ and $\lambda^{(L,R)}(x)$, respectively, in the diagram
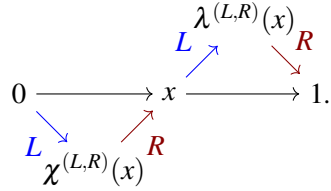
\[
\begin{tikzpicture}
    \node (A) at (-2,0){$0$};
    \node (B) at (0,0){$x$};
    \node (C) at (2,0){$1$.};
    
    \node (D) at (-1,-1){$\chi^{(L,R)}(x)$};
    \node (E) at (1, 1){$\lambda^{(L,R)}(x)$};

    \draw[->] (A) to (B);
    \draw[->] (B) to (C);
    \draw[->, blue] (A) --node[below left]{$L$} (D);
    \draw[->, DarkRed] (D) --node[below right]{$R$} (B);
    \draw[->, blue] (B) --node[above left]{$L$} (E);
    \draw[->, DarkRed] (E) --node[above right]{$R$} (C);
\end{tikzpicture}
\]
\end{defn}

In other words, $\chi^{(L,R)}(x)$ is the factoring object for the initial map $0\to x$ and $\lambda^{(L,R)}(x)$ is the factoring object for the terminal map $x\to 1$. It is a straightforward consequence of the defintions that both $\chi^{(L,R)}$ and $\lambda^{(L,R)}$ are monotone endomorphisms of $P$.

There is a natural anti-isomorphism of posets $\Fac(P)\to \Fac(P^{\op})$ given by $(L,R)\mapsto (R^\op,L^\op)$. It is an immediate consequence of the definitions that this duality plays nicely with characteristic and cocharacteristic functions:

\begin{prop}\label{prop:dualchar}
    Let $P$ be a finite lattice and let $(L,R)$ be a factorization system on $P$.  Then
\[
\begin{aligned}
  \chi^{(L,R)} &= \lambda^{(R^\op, L^\op)},\\
  \lambda^{(L,R)} &= \chi^{(R^\op, L^\op)}.
\end{aligned}
\]
\hfill\qedsymbol
\end{prop}

We can also compute (co)characteristic functions by solving certain optimization problems in our poset:

\begin{theorem}\label{thm:fmlae}
Suppose $P$ is a finite lattice and $(L,R)\in \Fac P$. Then
\[
\begin{aligned}
  \chi^{(L,R)}(x) &= \min\{y\in P\mid y~R~x\} = \max \{y\in P\mid 0~L~y\le x\},\\
  \lambda^{(L,R)}(x) &= \max\{y\in P\mid x~L~y\} = \min \{y\in P\mid x\le y~R~1\}.
\end{aligned}
\]
\end{theorem}
\begin{proof}
We will only check the identities involving $\lambda^{(L,R)}(x)$. The identities for $\chi^{(L,R)}(x)$ then follow by duality.

First note that $x~L~\lambda^{(L,R)}(x)$ by definition, and thus $\lambda^{(L,R)}(x)$ is an element of the set $\{y\in P\mid x~L~y\}$. Now if $x~L~y$, then we have the diagram
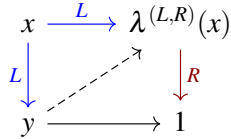
\[
\begin{tikzcd}
  x \arrow[blue]{r}{L} \arrow[blue,swap]{d}{L} & \lambda^{(L,R)}(x) \arrow[DarkRed]{d}{R}\\
  y \arrow{r} \arrow[dashed]{ur}& 1
\end{tikzcd}
\]
and the mandatory lift demonstrates that $y\le \lambda^{(L,R)}(x)$. We conclude that $\lambda^{(L,R)}(x)$ is in fact the unique maximum element of $\{y\in P\mid x~L~y\}$.

Next observe that $x~L~\lambda^{(L,R)}(x)~R~1$ by definition, and thus $\lambda^{(L,R)}(x)$ is an element of $\{y\in P\mid x\le y~R~1\}$. Now if $x\le y~R~1$, then we have the diagram
\[
\begin{tikzcd}
  x \arrow{r}\arrow[blue,swap]{d}{L} & y\arrow[DarkRed]{d}{R}\\
  \lambda^{(L,R)}(x) \arrow[DarkRed,swap]{r}{R} \arrow[dashed]{ur} & 1
\end{tikzcd}
\]
and the mandatory lift demonstrates that $\lambda^{(L,R)}(x)\le y$. We conclude that $\lambda^{(L,R)}(x)$ is the unique minimum element of $\{y\in P\mid x\le y~R~1\}$.
\end{proof}

We now consider the properties of the assignments
\[
\begin{aligned}
    \chi\colon \Fac P&\longrightarrow \End P\\
    (L, R)&\longmapsto \chi^{(L,R)}
\end{aligned}
\ \ \ \ \ \ \ \ 
\begin{aligned}
    \lambda\colon \Fac P&\longrightarrow \End P\\
    (L, R)&\longmapsto \lambda^{(L,R)}.
\end{aligned}
\]
We give $\Fac P$ the partial order defined so that $(L,R)\le (L',R')$ if and only if $R\subseteq R'$ (which is equivalent to $L'\subseteq L$). Meanwhile $\End P$ carries the pointwise partial order: $f\le g$ if and only if $f(x)\le g(x)$ in $P$ for all $x\in P$. We recall that an \emph{antitone} map of posets is one that reverses order relations.

\begin{prop}\label{prop:lambdaanti}
Both $\chi$ and $\lambda$ are antitone maps of posets.
\end{prop}
\begin{proof}
By careful consider of the consequences of \autoref{prop:dualchar}, it suffices to check that $\lambda$ is antitone. Suppose that $(L,R)\le (L',R')\in \Fac P$ and fix an arbitrary $x\in P$. Since $L'\subseteq L$, we know that $\{y\in P\mid x~L'~y\}\subseteq \{y\in P\mid x~L~y\}$. By the maximum formulation of $\lambda$ in \autoref{thm:fmlae}, it follows that
\[
  \lambda^{(L',R')}(x)\le \lambda^{(L,R)}(x),
\]
whence $\lambda^{(L',R')}\le \lambda^{(L,R)}$, as desired.
\end{proof}

We now aim to understand the images and fibers of $\chi$ and $\lambda$, which will require the introduction of some additional terminology.

\begin{defn}\label{def:int_clo}
A \textbf{closure operator} on a poset $P$ is a monotone function $c\colon P\to P$ that is \textbf{extensive} ($x\le cx$) and \textbf{idempotent} ($ccx=cx$). We write $\clop P$ for the poset of closure operators under the pointwise partial order.

An \textbf{interior operator} on a poset $P$ is a monotone function $i\colon P\to P$ that is \textbf{contractive} ($ix\le x$) and idempotent. We write $\End^\circ P$ for the poset of interior operators under the pointwise partial order.
\end{defn}

\begin{defn}\label{defn:reflective}
    For a category $\mathbf{C}$ with terminal object $1$, we say an orthogonal factorization system $(L, R)$ on $\mathbf C$ is a \textbf{reflective} factorization system if $R/1$ is a reflective subcategory of $\mathbf{C}$; that is, the inclusion functor $R/1\to \mathbf{C}$ admits a left adjoint.

    Dually, $(L,R)$ is a \textbf{coreflective} factorization system if $0\backslash L$ is a coreflective subcategory of $\mathbf C$; that is, the inclusion functor $0\backslash L\to \mathbf C$ admits a right adjoint.
\end{defn}

\begin{remark}\label{rmk:joinmeet}
When $\mathbf C$ is a poset category, there are simpler formulations of (co)reflectivity. In particular, a consequence of the adjoint functor theorem is that a functor from a lattice to a poset is a left (respectively right) adjoint if and only if it preserves joins (respectively meets). In particular, a factorization system $(L,R)$ on a lattice is reflective if and only if the objects of $R/1$ (\emph{i.e.}, $\{x\in P\mid x~R~1\}$) are closed under meet in $P$. Similarly, $(L,R)$ is coreflective if and only if the objects of $0\backslash L$ (\emph{i.e.}, $\{x\in P\mid 0~L~x\}$) are closed under join in $P$.

We will see in \autoref{thm:submonoid} that $(L,R)\mapsto R/1$ induces a bijection between reflective factorization systems and submonoids of $(P,\wedge)$, while $(L,R)\mapsto 0\backslash L$ induces a bijection between coreflective factorization system and submonoids of $(P,\vee)$. We will use this fact in a crucial step of our proof of \autoref{thm:fibers}, but of course its proof is independent.
\end{remark}

This allows us to state our main theorem. For our purposes, a (bounded, closed) \textbf{interval} in a poset $P$ will be any induced subposet of the form $[a,b] = \{x\in P\mid a\le x\le b\}$.

\begin{theorem}\label{thm:fibers}
Suppose $P$ is a finite lattice. Then $\chi\colon \Fac P\to \End P$ has image $\End^\circ P$. The fibers of $\chi$ are intervals in $\Fac P$ and restricting $\chi$ to maximal elements of fibers produces a bijection between coreflective factorization systems and interior operators.

Dually, $\lambda\colon \Fac P\to \End P$ has image $\clop P$. The fibers of $\lambda$ are intervals in $\Fac P$ and restricting $\lambda$ to minimal elements of fibers produces a bijection between reflective factorization systems and closure operators.
\end{theorem}

We prove \autoref{thm:fibers} in several steps, beginning with a formal lemma.

\begin{lemma}\label{lemma:refcoref}
A factorization system $(L,R)$ on $\mathbf C$ is reflective if and only if $(R^\op,L^\op)$ is coreflective on $\mathbf C^\op$. \hfill\qedsymbol
\end{lemma}

As a consequence of \autoref{lemma:refcoref} and \autoref{prop:dualchar}, we only need to prove the $\lambda$ portion of \autoref{thm:fibers}; the portion about $\chi$ then follows by duality.

The following two lemmas will combine to show that the fibers of $\lambda$ are intervals.

\begin{lemma}\label{lemma:fiberclosemeet}
    Let $P$ be a lattice, and let $(L,R), (L', R')$ be factorization systems on $P$ such that $\lambda^{(L,R)}=\lambda^{(L',R')}$.  Then $\lambda^{(L, R) \wedge (L',R')}=\lambda^{(L,R)}$.
\end{lemma}

\begin{proof}
     We first note that $(L, R)\wedge (L', R')=({}^\bot (R\cap R'), R\cap R')$. Thus, we have by \autoref{prop:lambdaanti} that $\lambda^{(L, R) \wedge (L',R')} \geq \lambda^{(L,R)}$. Let $x\in P$ and $y=\lambda^{(L,R)}(x)=\lambda^{(L',R')}(x)$.  Since $y~R~1$ and $y~R'~1$, we have that $y~R\wedge R'~1$.  Thus $\lambda^{(L, R) \wedge (L',R')} \leq \lambda^{(L,R)}$ and $\lambda^{(L, R) \wedge (L',R')}= \lambda^{(L,R)}$.
\end{proof}

\begin{lemma}\label{lemma:fiberclosejoin}
    Let $P$ be a lattice, and let $(L,R),(L',R')$ be factorization systems on $P$ such that $\lambda^{(L,R)}=\lambda^{(L',R')}$.  Then $\lambda^{(L,R) \vee (L',R')}=\lambda^{(L,R)}$.
\end{lemma}

\begin{proof}
    Suppose by way of contradiction that $\lambda^{(L, R) \vee (L',R')}\neq \lambda^{(L,R)}$.  Since $(L, R) <(L, R)\vee (L', R')$, by \autoref{prop:lambdaanti} it follows that $\lambda^{(L, R) \vee (L',R')}< \lambda^{(L,R)}$.  Consequently,  there is a $x\in P$ where $\lambda^{(L, R) \vee (L',R')}(x)<\lambda^{(L,R)}(x)$. Let $y':=\lambda^{(L, R) \vee (L',R')}(x)$ and $y:=\lambda^{(L,R)}(x)$.
    
    It must be the case that $y'~\not{ R}~1$, or else $\lambda^{(L,R)}(x)$ would not be $y$, and since $\lambda^{(L',R')}(x)=y$, we also have that $y'~\not{ R'}~1$. Thus the relation $y'~R\vee R'~1$ is a new relation not a part of $R, R'$ and by Theorem A.2 \cite{Rubin}, the only new relations of $R\vee R'$ are due to transitivity. Without loss of generality, suppose that  there were a $w\in P$ where $y'~R~w~R'~1$.  
    
    Note that if there were a $w'\in P$ such that $y'\leq w'~R'~1$, then by closure under pullback we would have that $w\wedge w'~R'~w$ and thus $w\wedge w'~R'~1$ (see \autoref{fig:fiberclosejoin}).  Then, since $y'\leq w\wedge w'$, we have by pullback again that $y'~R~w\wedge w'$ (also depicted in \autoref{fig:fiberclosejoin}).

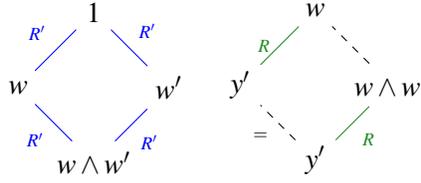
\begin{figure}
    \[\begin{tikzpicture}
    \node (A) at (1,2){$1$};
    \node (B) at (0,1){$w$};
    \node (C) at (2,1){$w'$};
    
    \node (E) at (1, 0){$w\wedge w'$};

    \draw[blue] (A) --node[above left]{\tiny$R'$} (B);
    \draw[blue] (A) --node[above right]{\tiny$R'$} (C);

    \draw[blue] (E) --node[below left]{\tiny $R'$} (B);
    \draw[blue] (E) --node[below right]{\tiny$R'$} (C);
\end{tikzpicture}
\ \ \ \ 
\begin{tikzpicture}
    \node (B) at (1,1){$w$};
    
    \node (D) at (0,0){$y'$};
    \node (E) at (2, 0){$w\wedge w'$};
    \node (F) at (1, -1){$y'$};

    \draw[ForestGreen] (D) --node[ left]{\tiny $R$} (B);
    \draw[dashed] (E) -- (B);
    \draw[dashed] (F) --node[below left]{\tiny $=$} (D);
    \draw[ForestGreen] (F) --node[below right]{\tiny $R$} (E);
\end{tikzpicture}
\]
\caption{Relations invoked in the proof of \autoref{lemma:fiberclosejoin}.}\label{fig:fiberclosejoin}
\end{figure}

    Thus, without loss of generality, we may let $w$ be the minimal such choice and $w=\lambda^{(L',R')}(y')$.  But then we have that $w=\lambda^{(L,R)}(y')$, and so $w~R~1$, and by transitivity $y'~R~1$.  This contradicts the claim that $\lambda^{(L, R) \vee (L',R')}(x)<\lambda^{(L,R)}(x)$.
\end{proof}

In the following lemma, we come to terms with the minima of the fibers of $\lambda$.

\begin{lemma}\label{lemma:fiberreflect}
Suppose $P$ is a finite lattice and $(L,R)$ is the lower endpoint (\emph{i.e.}, minimum) of a fiber of $\lambda\colon \Fac P\to \End P$. Then $(L,R)$ is reflective.
\end{lemma}
\begin{proof}
The cocharacteristic function on $(L,R)$ is completely determined by $R/1$. Now observe that the minimal factorization system whose right morphisms contain a given $R/1$ is necessarily reflective.
\end{proof}

We now start our analysis of the image of $\lambda$.

\begin{lemma}
Let $P$ be a finite lattice and $(L,R)$ a factorization system on $P$. Then $\lambda^{(L,R)}$ is a closure operator.
\end{lemma}
\begin{proof}
We already know that $\lambda^{(L,R)}\colon P\to P$ is monotone, so it suffices to check that it is extensive and idempotent. By \autoref{thm:fmlae}, we know that $\lambda^{(L,R)}(x) = \max\{y\in P\mid x~L~y\} \ge x$, so $\lambda^{(L,R)}$ is extensive. Using the same formula, transitivity of $L$ implies that $\lambda^{(L,R)}$ is idempotent.
\end{proof}

Finally, we are in a position to complete the proof of the main theorem of this section.

\begin{proof}[Proof of \autoref{thm:fibers}]
We have previously noted that the theorem on $\chi$ follows by duality from the theorem for $\lambda$.

The preceding lemmas demonstrate that $\lambda\colon \Fac P\to \End P$ factors through $\clop P$ and its fibers are intervals whose lower endpoints are reflective factorization systems. Thus it suffices to check that the restriction of $\lambda$ to reflective factorization systems is a bijection onto $\clop P$.

Consider the diagram
\[
\begin{tikzcd}
  \Fac^{\refl} P \arrow{r}{\lambda} \arrow{dr} & \clop P\arrow{d}\\
  & \SubMon(P,\wedge)
\end{tikzcd}
\]
where the vertical arrow takes $c\in \clop P$ to the fixed points of $c$, $P^c := \{x\in P\mid cx=x\}$, and the diagonal arrow takes $(L,R)$ to (the set of objects of) $R/1$. The reader may check that this diagram commutes. We will see in \autoref{thm:clsubmon} that the vertical arrow is a bijection, and \autoref{thm:submonoid} tells us the diagonal arrow is a bijection. The results follows.
\end{proof}

\section{The unbearable ubiquity of (co)reflective factorization systems}\label{sec:crypto}
In this section, we establish a web of bijections --- summarized in \autoref{fig:web} --- between (co)reflective factorization systems and other structures.

\begin{figure}[h]
\begin{tikzpicture}[scale=2.75]
        \node (A1) at (-0.75,-1){$\coreft(P)$};
        \node (B1) at (0.75,-1){$\End^\circ(P)$};
        \node (C1) at (0,-0.25){$\SubMon(P, \vee)$};
        \node (D1) at (-2,-1.5){$\text{Fibrant}(P)$};
        \node (E1) at (-2,-1){$\sat(P)$};
        \node (F1) at (-2,-0.5){$\{\max\{\chi^{-1}(c)\} \}$};
        \node (G1) at (2,-1){$\text{coMonad}(P)$};

        \node (A2) at (-0.75,1){$\reft(P)$};
        \node (B2) at (0.75,1){$\overline{\End}(P)$};
        \node (C2) at (0,0.25){$\SubMon(P, \wedge)$};
        \node (D2) at (-2,1.5){$\text{coFibrant}(P)$};
        \node (E2) at (-2,1){$\cosat(P)$};
        \node (F2) at (-2,0.5){$\{\min\{\lambda^{-1}(i)\} \}$};
        \node (G2) at (2,1){$\text{Monad}(P)$};

        \draw[<->] (A1) --node[above , sloped, ]{\tiny  \ref{thm:refdisk}} (E1);
        \draw[<->] (A1) --node[above , sloped, ]{\tiny   \ref{prop:model}} (D1);
        \draw[<->] (A1) --node[above , sloped, ]{\tiny  \ref{lemma:fiberreflect}} (F1);
        \draw[<->] (A1) --node[above , sloped, ]{\tiny  \ref{thm:fibers}} (B1);
        \draw[<->] (A1) --node[below , sloped, ]{\tiny  \ref{thm:submonoid}} (C1);
        \draw[<->] (C1) --node[below , sloped, ]{\tiny  \ref{thm:clsubmon}} (B1);
        \draw[<->] (B1) --node[above , sloped, ]{\tiny  \ref{prop:monad}} (G1);

        \draw[<->] (A2) --node[below , sloped, ]{\tiny  \ref{thm:refdisk}} (E2);
        \draw[<->] (A2) --node[below , sloped, ]{\tiny  \ref{prop:model}} (D2);
        \draw[<->] (A2) --node[below , sloped, ]{\tiny  \ref{lemma:fiberreflect}} (F2);
        \draw[<->] (A2) --node[below , sloped, ]{\tiny  \ref{thm:fibers}} (B2);
        \draw[<->] (A2) --node[above , sloped, ]{\tiny  \ref{thm:submonoid}} (C2);
        \draw[<->] (C2) --node[above , sloped, ]{\tiny  \ref{thm:clsubmon}} (B2);
        \draw[<->] (B2) --node[below , sloped, ]{\tiny  \ref{prop:monad}} (G2);

        \draw[dashed] (-2.5,0)--(2.5,0);

    \end{tikzpicture}

\caption{Above the dashed line: Bijections between reflective factorization systems and other mathematical structures. Below the dashed line: The same, but for coreflective factorization systems. Each arrow is labeled by the number of the theorem (or proposition, \emph{etc}.) establishing the bijection.  The dashed line represents the formal duality between factorization systems $(L, R)$ on $P$ and $(R^\op, L^\op)$ on $P^\op$.}\label{fig:web}
\end{figure}
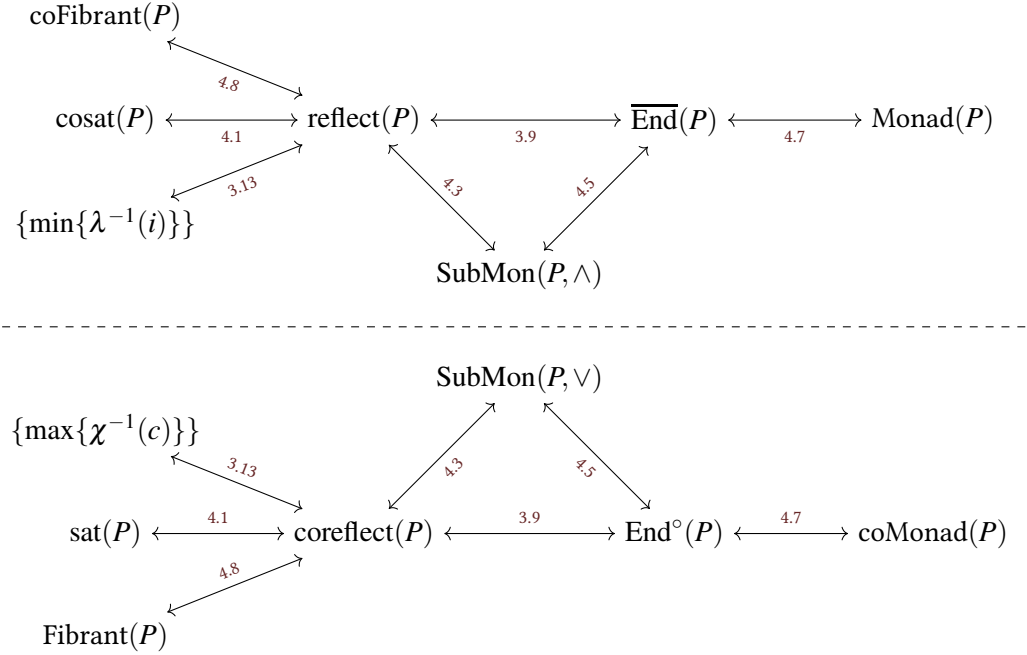

Given a factorization system $(L,R)$ on a lattice $P$, we abuse notation and write $R/1$ and $0\backslash L$ for the objects of the these categories, \emph{i.e.}, $R/1 := \{x\in P\mid x~R~1\}$ and $0\backslash L := \{x\in P\mid 0~L~x\}$.

\begin{theorem}\label{thm:refdisk}
Let $P$ be a finite lattice. Under the bijection $\Fac P\cong \Tr P$ of \autoref{thm:fooqw}, reflective factorization systems correspond to disklike transfer systems, and coreflective factorization systems correspond to saturated transfer systems.
\end{theorem}
\begin{proof}[Proof sketch]
A factorization system $(L,R)$ is reflective if and only if it is determined by $R/1$, in which case $\langle R/1\rangle = R$ as a transfer system.

Similarly, $(L,R)$ is coreflective if and only if it is determined by $0\backslash L$, in which case the ``co-transfer system'' $L$ is generated by $0\backslash L$. The reader may check that this minimality condition on the left set corresponds to the maximality condition of saturation on the right set.
\end{proof}
\begin{corollary}\label{cor:satdisk}
Let $P$ be a finite lattice. Then under the duality
\[
\begin{aligned}
  \Tr P&\longrightarrow \Tr P^\op\\
  R&\longmapsto ({}^\perp R)^\op,
\end{aligned}
\]
saturated transfer systems on $P$ correspond to disklike factorization systems on $P^\op$.
\end{corollary}
\begin{proof}
This is the duality between coreflective and reflective factorization systems transported to transfer systems via \autoref{thm:refdisk}.
\end{proof}

It is also the case that (co)reflective factorization systems encode certain submonoids of the commuatative idempotent monoids $(P,\vee)$ and $(P,\wedge)$. The following proof appears in unpublished work of Balchin--MacBrough--Ormsby.

\begin{theorem}\label{thm:submonoid}
The assignment $(L,R)\mapsto R/1$ is a bijection between reflective factorization systems and submonoids of $(P,\wedge)$; dually, $(L,R)\mapsto 0\backslash L$ is a bijection between coreflective factorization systems and submonoids of $(P,\vee)$.
\end{theorem}
\begin{proof}
By duality, we only need to handle the case of reflective factorization systems. We define an adjoint pair of functors (\emph{i.e.}, monotone Galois connection) $(F,G)\colon 2^P\leftrightarrows \Fac P$ between the lattice of subsets of $P$ and factorization systems on $P$ where
\[
  G(L,R) = \{x\in P\mid x~R~1\}.
\]
and
\[
  FS = ({}^\perp \langle S/1\rangle,\langle S/1\rangle).
\]
Since Galois connections induce Galois correspondences on images (\cite[Theorem 2]{ore}), we only need to identify the images of $G$ and $F$.

Clearly $1\in G(L,R)$, and since the right morphisms in a factorization system are stable under pullback in $P$, we see that $G(L,R)$ is also closed under the operation $\wedge$, making it a submonoid of $(P,\wedge)$. Given a submonoid $M$ of $(P,\wedge)$, one may check that
\[
  ({}^\perp \langle M/1\rangle, \langle M/1\rangle)
\]
maps to $M$ under $G$, so the image of $G$ is exactly the submonoids of $(P,\wedge)$.

Finally, by \autoref{thm:refdisk}, the image of $F$ is exactly the reflective factorization systems on $P$.
\end{proof}
\begin{remark}
Submonoids of $(P,\wedge)$ are sometimes called \textbf{Moore families}. Additionally, when $P = 2^S$ is the powerset lattice of a set $S$, such submonoids are sometimes called \textbf{closure systems} on $S$.

Closure systems are notoriously challenging to enumerate, and the total number of them is only known for $0\le \#S\le 7$. See OEIS A102896 and the papers \cite{Higuchi,HabibNourine,countingmoore}. Work of Kleitman \cite{Kleitman} shows that the base-2 logarithm of the number of closure systems on a cardinality $n$ set is asymptotic to the middle binomial coefficient $\binom{n}{\lfloor n/2\rfloor}$.
\end{remark}

The following result seems to be well-known, but is hard to track down in the literature.

\begin{theorem}\label{thm:clsubmon}
There is a bijection between $\clop P$ and $\SubMon(P,\wedge)$ given by $c\mapsto P^c = \{x\in P\mid cx=x\}$. Dually, there is a bijection between $\End^\circ P$ and $\SubMon(P,\vee)$ given by $i\mapsto P^i$.
\end{theorem}
\begin{proof}
By duality, we only need to check the result on closure operators. We first check that if $c\in \clop P$, then $P^c$ is a submonoid of $(P,\wedge)$. Suppose $x,y\in P^c$ so that $cx=x,cy=y$. Since $x\wedge y\le x,y$ and $c$ is monotone, we have $c(x\wedge y)\le cx=x, cy=y$, so $c(x\wedge y)\le x\wedge y$. Suppose we have some other $z\le x,y$. Then $z\le x\wedge y$, so $cz\le c(x\wedge y)$. This demonstrates that $c(x\wedge y)$ is the greatest lower bound of $x,y$, \emph{i.e.},
\[
  c(x\wedge y) = x\wedge y,
\]
whence $P^c$ is a submonoid of $(P,\wedge)$.

We now construct an inverse to $c\mapsto P^c$. Given a submonoid $A$ of $(P,\wedge)$, define $f_A\colon P\to P$ by $f_Ax = \min\{y\in A\mid x\le y\}$. One may check by first principles that $f_A$ is monotone, extensive, and idempotent, so $f_A$ is a closure operator.  We leave it to the reader to check that for each closure operator $c$, $f_{P^c} = c$, and for each submonoid $A\le (P,\wedge)$, $P^{f_A} = A$.
\end{proof}

\begin{remark}\label{rmk:pB}
Don Knuth \cite{knuth:parades} has observed that subsemigroups of $([m]\times[n],\vee)$ are counted by the poly-Bernoulli number
\[
  B_{m+1,n+1} = \sum_{k\ge 0}(k!)^2{m+2\brace k+1}{n+2\brace k+1}.
\]
Since there are half as many submonoids as subsemigroups (via the forgetful and ``remove 0'' maps), we learn that
\[
  \#\SubMon([m]\times[n],\vee) = \frac{1}{2}B_{m+1,n+1}.
\]
Via Kaneko's original work \cite{kaneko1997poly} on poly-Bernoulli numbers, this also recovers the asymmetric formula for saturated transfer systems on $[m]\times[n]$ in the work of Hafeez--Marcus--Ormsby--Osorno \cite{hmoo}. Forthcoming work of the 2024 Reed Collaborative Mathematics Group \cite{cmrg2024} extends these enumerations to more general modular lattices.
\end{remark}

We conclude with two relatively minor observations which may nonetheless prove useful in future enumeration efforts. The first relates to (co)monads on a (category induced by a) finite lattice.

\begin{prop}\label{prop:monad}
Let $P$ be a lattice, which we also view as a bicomplete category. There is a bijection between monads on $P$ and closure operators on $P$; dually, there is a bijection between comonads on $P$ and interior operators on $P$.
\end{prop}
\begin{proof}
By duality, we only need to demonstrate the result for monads. A monad consists of an endofunctor $T\colon P\to P$ and natural transformations $\eta\colon \id_P\to T$ and $\mu\colon T\circ T\to T$. An endofunctor on $P$ is the same data as a monotone function on $P$, the natural transformation $\eta$ witnesses $\eta_x\colon x\le Tx$ (extensivity), and the natural transformation $\mu$ implies idempotence when combined with extensivity: \[\mu_x\colon TTx\le Tx\implies Tx\le TTx\le Tx\implies TTx=Tx.\]
\end{proof}

Finally, (co)reflective factorization systems may also be viewed as special types of model structures. Recall that in the Joyal--Tierney \cite{JT07} presentation of a model structure is an interval of weak factorization systems $(C,AF)\le (AC,F)$ for which $W := AF\circ AC$ satisfies 3-for-2. Call a model structure \textbf{fibrant} when $F = \mathrm{All}$, and call it \textbf{cofibrant} when $C = \mathrm{All}$.

\begin{prop}\label{prop:model}
Let $P$ be a finite poset. Coreflective factorization systems on $P$ are in bijection with fibrant model structures on $P$ via
\[
  (L,R)\longmapsto [(L,R)\le (\mathrm{Id}, \mathrm{All})].
\]
Reflective factorization systems on $P$ are in bijection with cofibrant model structures on $P$ via
\[
  (L,R)\longmapsto [(\mathrm{All},\mathrm{Id})\le (L,R)].
\]
\end{prop}
\begin{proof}
Suppose $(L,R)$ is coreflective. By \autoref{thm:refdisk}, we know that $R$ is a saturated transfer system, so $R\circ \mathrm{Id} = R$ satisfies 3-for-2. This makes $[(L,R)\le (\mathrm{Id}, \mathrm{All})]$ a fibrant model structure on $P$, and it is straightforward to check that this assignment may be reversed. The remainder of the theorem is dual.
\end{proof}

\bibliography{closure}
\bibliographystyle{alpha}

\end{document}